\documentclass[11 pt]{amsart}

\usepackage[english]{babel}
\usepackage{amssymb}
\usepackage{amsfonts}
\usepackage{amsmath}
\usepackage{amsthm}
\usepackage{epsfig}
\usepackage{color}
\usepackage{amscd}
\usepackage[all]{xy}
\usepackage{hyperref}
\usepackage{placeins}
\usepackage{caption}
\newtheorem{lemma}{Lemma}[section]
\newtheorem{thm}[lemma]{Theorem}

\newtheorem{cor}[lemma]{Corollary}

\theoremstyle{definition}

\theoremstyle{remark}

\newtheoremstyle{letter}{}{}{\itshape}{}{\bfseries}{.}{.5em}{#1\thmnote{ #3}}
\theoremstyle{letter}

\DeclareMathOperator{\sech}{sech}

\usepackage[letterpaper,left=0.9in,right=0.9in,top=0.95in,bottom=0.95in]{geometry}
\author{\tiny{Maria Dostert}}
\address{\tiny{Royal Institute of Technology (KTH),  Department of Mathematics, SE-100 44 Stockholm, Sweden}}
\email{\tiny{dostert@kth.se}}

\author{\tiny{Alexander Kolpakov}}
\address{\tiny{\parbox{\linewidth}{Institut de Math\'ematiques, Universit\'e de Neuch\^atel, 2000 Neuch\^atel, Suisse/Switzerland \\
Laboratory of combinatorial and geometric structures, Moscow Institute of Physics \\ and Technology, Dolgoprudny, Russia}}}
\email{\tiny{kolpakov.alexander@gmail.com}}

\title[Kissing number in non-Euclidean spaces of constant sectional curvature]{Kissing number in non-Euclidean spaces \\ of constant sectional curvature}

\begin{document}

\begin{abstract}
This paper provides upper and lower bounds on the kissing number of congruent radius $r > 0$ spheres in hyperbolic $\mathbb{H}^n$ and spherical $\mathbb{S}^n$ spaces, for $n\geq 2$.  For that purpose, the kissing number is replaced by the kissing function $\kappa_H(n, r)$, resp. $\kappa_S(n, r)$, which depends on the dimension $n$ and the radius $r$.

After we obtain some theoretical upper and lower bounds for $\kappa_H(n, r)$, we study their asymptotic behaviour and show, in particular, that $\kappa_H(n,r) \sim (n-1) \cdot d_{n-1} \cdot B(\frac{n-1}{2}, \frac{1}{2}) \cdot e^{(n-1) r}$, where $d_n$ is the sphere packing density in $\mathbb{R}^n$, and $B$ is the beta-function.  Then we produce numeric upper bounds by solving a suitable semidefinite program, as well as lower bounds coming from concrete spherical codes. A similar approach allows us to locate the values of $\kappa_S(n, r)$, for $n= 3,\, 4$, over subintervals in $[0, \pi]$ with relatively high accuracy. \\

\noindent
\textit{Key words: } hyperbolic geometry, spherical geometry, kissing number, semidefinite programming.\\
\noindent
\textit{2010 AMS Classification: } Primary: 05B40. Secondary: 52C17, 51M09. \\
\end{abstract}

\maketitle
\vspace*{-0.3in}

\section{Introduction}
The kissing number $\kappa(n)$ is the maximal number of unit spheres that can simultaneously touch a central unit sphere in $n$-dimensional Euclidean space $\mathbb{R}^n$ without pairwise overlapping.  The research on the kissing number leads back to $1694$, when Isaac Newton and David Gregory had a discussion whether $\kappa(3)$ is equal to $12$ or $13$ \cite{Casselman}.

The exact value of $\kappa(n)$ is only known for  $n = 1, 2, 3, 4, 8, 24$, whereas for $n = 1, 2$ the problem is trivial. In 1953, Sch\"utte and van der Waerden proved that $\kappa(3) = 12$ \cite{SchuetteWaerden}. Furthermore, Delsarte, Goethals, and Seidel \cite{DelsarteGoethalsSeidel} developed a linear programming (LP) bound, which was used by Odlyzko and Sloane \cite{Odlyzko79newbounds} and independently by Levenshtein \cite{Levenshtein} to prove $\kappa(8) = 240$ and $\kappa(24) = 196560$. Later, Musin \cite{Musin} showed that $\kappa(4) = 24$ by using a stronger version of the LP bound. Bachoc and Vallentin \cite{BachocVallentin} strengthened the LP bound further by using a semidefinite program (SDP), which is a generalisation of a linear program. In \cite{MittelmannVallentin}, Mittelmann and Vallentin give a table with the best upper bounds for the kissing number for $n \leq 24$ by using SDP bounds. Moreover, Machado and Oliveira \cite{Machado} improved some of these results, by exploiting polynomial symmetry in the SDP.

The study of the Euclidean kissing number and the contact graphs of kissing configurations in $\mathbb{R}^n$ is still an active area in geometry and optimisation, with a few recent developments cf. \cite{Bezdek1, Bezdek-Reid, DKO, Glazyrin}. 

In this paper, we consider an analogous problem in $n$-dimensional hyperbolic space $\mathbb{H}^n$, as well as in $n$-dimensional spherical space $\mathbb{S}^n$. Some other non-Euclidean geometries, notably the ones on Thurston's list of the eight $3$-dimensional model spaces, have been studied in \cite{Szirmai1, Szirmai2}.

By a sphere of radius $r$ in the hyperbolic space $\mathbb{H}^n$ we mean the set of points at a given geodesic distance $r>0$ from a specific point, which is called the centre of the sphere. Thus, we avoid considering horospheres (i.e. spheres centred at the ideal boundary $\partial \mathbb{H}^n$) and hyperspheres (i.e. geodesic planes dual to the so-called ultra-ideal points of $\mathbb{H}^n$). Packings of $\mathbb{H}^n$ by horospheres, especially in higher  dimensions $n\geq 4$, recently received some attention in \cite{Szirmai3}.

In a kissing configuration every sphere in $\mathbb{H}^n$ has the same radius. Unlike in Euclidean spaces, the kissing number in $\mathbb{H}^n$ depends on the radius $r$, and we denote it by $\kappa_H(n,r)$. By using a purely Euclidean picture of the arrangement of $k$ spheres of radius $r$ in $\mathbb{H}^n$ seen in the Poincar\'e ball model, the kissing number in $\mathbb{H}^n$ coincides with the maximal number of spheres of radius $\frac{1}{2} \left(\tanh\frac{3r}{2} - \tanh\frac{r}{2}\right)$ that can simultaneously touch a central sphere of radius $\tanh\frac{r}{2}$ in $\mathbb{R}^n$ without pairwise intersecting. 

The kissing number in $\mathbb{S}^n$ is denoted by $\kappa_S(n, r)$, and also depends on the radius $r$ of the spheres in the kissing configuration, although here $r$ belongs to the bounded interval $(0, \pi]$. Since $\mathbb{S}^n$ is a compact metric space, $\kappa_S(n, r)$ is a decreasing function of $r$, while  $\kappa_H(n, r)$ increases with $r$ exponentially fast.

In Section \ref{theoreticalbound-H}, we develop some upper and lower bounds for the kissing number in $\mathbb{H}^n$, $n \geq 2$, in order to evaluate its asymptotic behaviour. In Section \ref{sdpbound}, we adapt the SDP by Bachoc and Vallentin to obtain upper bounds for $\kappa_H(n,r)$ and $\kappa_S(n,r)$. In Section \ref{configurations}, concrete kissing configurations are discussed, which provide lower bounds for $\kappa_H(n,r)$. A great deal of them is taken from the spherical codes produced and collected by Hardin, Smith, and Sloane \cite{Sloane-et-al}, which often turn out to be optimal. 

For certain dimensions and radii, we compute the lower bounds by using the results of Section \ref{theoreticalbound-H}, and the upper bounds due to Levenshtein \cite{Levenshtein} and Coxeter \cite{Bor, Coxeter, FT}\footnote{This upper bound was conjectured by Coxeter in \cite{ Coxeter} and, even earlier, by Fejes T\'oth in \cite{FT}. It was proved much later by B\"or\"oczky in \cite{Bor}.}. Then, we compare them with the lower bounds given by concrete kissing configurations and with the SDP upper bounds, respectively. 

In Section \ref{theoreticalbound-S}, some geometric upper and lower bounds are given for the kissing number $\kappa_S(n, r)$ in $\mathbb{S}^n$, $n\geq 2$.  It is easy to see that $\kappa(1) = 2$ for $\mathbb{R}^1$ and $\kappa_H(1, r) = 2$, for all $r>0$, in $\mathbb{H}^1$ (the latter being isometric to $\mathbb{R}^1$). However, already for $\mathbb{S}^1$ we have that $\kappa_S(1,r) = 2$, if $r \leq \frac{\pi}{3}$, $\kappa_S(1,r) = 1$, if $\frac{\pi}{3} < r \leq \pi$, and $\kappa_S(1,r) = 0$, if $\pi < r \leq 2\pi$. In Sections \ref{sdpbound} and \ref{configurations}, we manage to locate the values of $\kappa_S(n, r)$ over subintervals in $[0, \pi]$: for $n = 3$ this follows from the solution of Tammes' problem, and for $n = 4$ we use the same approach to obtain the lower and upper bounds as in hyperbolic space.

All the numerical results are collected in Section~\ref{section:upper-lower-bounds}. Finally, the average kissing number (which happens to coincide in the Euclidean and non-Euclidean settings) is discussed in Section~\ref{section:avg-kissing-number}. 
 
\section*{Acknowledgements}
\noindent
{\small The idea of the project was conceived during the workshop ``Discrete geometry and automorphic forms'' at the American Institute of Mathematics (AIM) in September, 2018. The authors are grateful to the workshop organisers, Henry Cohn (Microsoft Research -- New England) and Maryna Viazovska (EPF Lausanne), all its participants and the AIM administration for creating a unique and stimulating research environment. They would also like to thank Matthew de Courcy--Ireland (EPF Lausanne), Fernando M. de Oliveira (TU Delft), Oleg Musin (University of Texas Rio Grande Valley), and Alexander Barg (University of Maryland) for useful comments and inspiring discussions. M.D. was partially supported by the Wallenberg AI, Autonomous Systems and Software Program (WASP) funded by the Knut and Alice Wallenberg Foundation. A.K. was partially supported by the Swiss National Science Foundation (project no.~PP00P2-170560) and the Russian Federation Government (grant no. 075-15-2019-1926).}

\section{Estimating the kissing number in hyperbolic space}\label{theoreticalbound-H}

In this section we shall produce some upper and lower bounds for the kissing function $\kappa_H(n, r)$, so that we can analyse its asymptotic behaviour in the dimension $n$, and in the radius $r$, for large values of the respective parameters. We refer the reader to \cite[\S~2.1]{Ratcliffe} and \cite[\S~4.5]{Ratcliffe} for all the necessary elementary facts about hyperbolic and spherical geometry, and their models. 

\subsection{Upper bound}

First we prove the following upper bound for $\kappa_H(n, r)$. The principal tool in the proof is passing to an essentially Euclidean setting and then using the classical ``spherical cap'' method, cf. \cite[\S~X.50]{Toth-RegularFigures}.

\begin{thm}\label{thm:upper-bound-h}
For any integer $n \geq 2$ and a non-negative number $r\geq 0$, we have that 
\begin{equation*}
\kappa_H(n, r) \leq \frac{2\, B\left( \frac{n-1}{2}, \frac{1}{2} \right)}{B\left( \frac{\sech^2 r}{4}; \frac{n-1}{2}, \frac{1}{2} \right)},
\end{equation*} 
where $B(x; y, z) = \int^x_0 t^{y-1} (1-t)^{z-1} dt$, for all $x \in [0,1]$ and $y, z > 0$, is the incomplete beta-function, and $B(y, z) = B(1; y, z)$. 
\end{thm}
\begin{proof}
In the proof we shall use a purely Euclidean picture of the arrangement of $k = \kappa_H(n, r)$ spheres in the hyperbolic space $\mathbb{H}^n$ seen in the 
Poincar\'e ball model, c.f. \cite[\S 4.5]{Ratcliffe}. A part of the kissing configuration looks (through the Euclidean eyes) as depicted in Fig.~\ref{fig:spheres1}. Let $S_0$ be the ``large'' sphere centred at $O$ of Euclidean radius $r_1 = \tanh \frac{r}{2}$ that is in the kissing configuration with $k$ ``small'' spheres $S_1$, $\dots$, $S_k$ of equal Euclidean radii $r_2 = \frac{1}{2} \left(\tanh\frac{3r}{2} - \tanh\frac{r}{2}\right)$. 

\begin{figure}[h]
\centering
\includegraphics[scale=0.4]{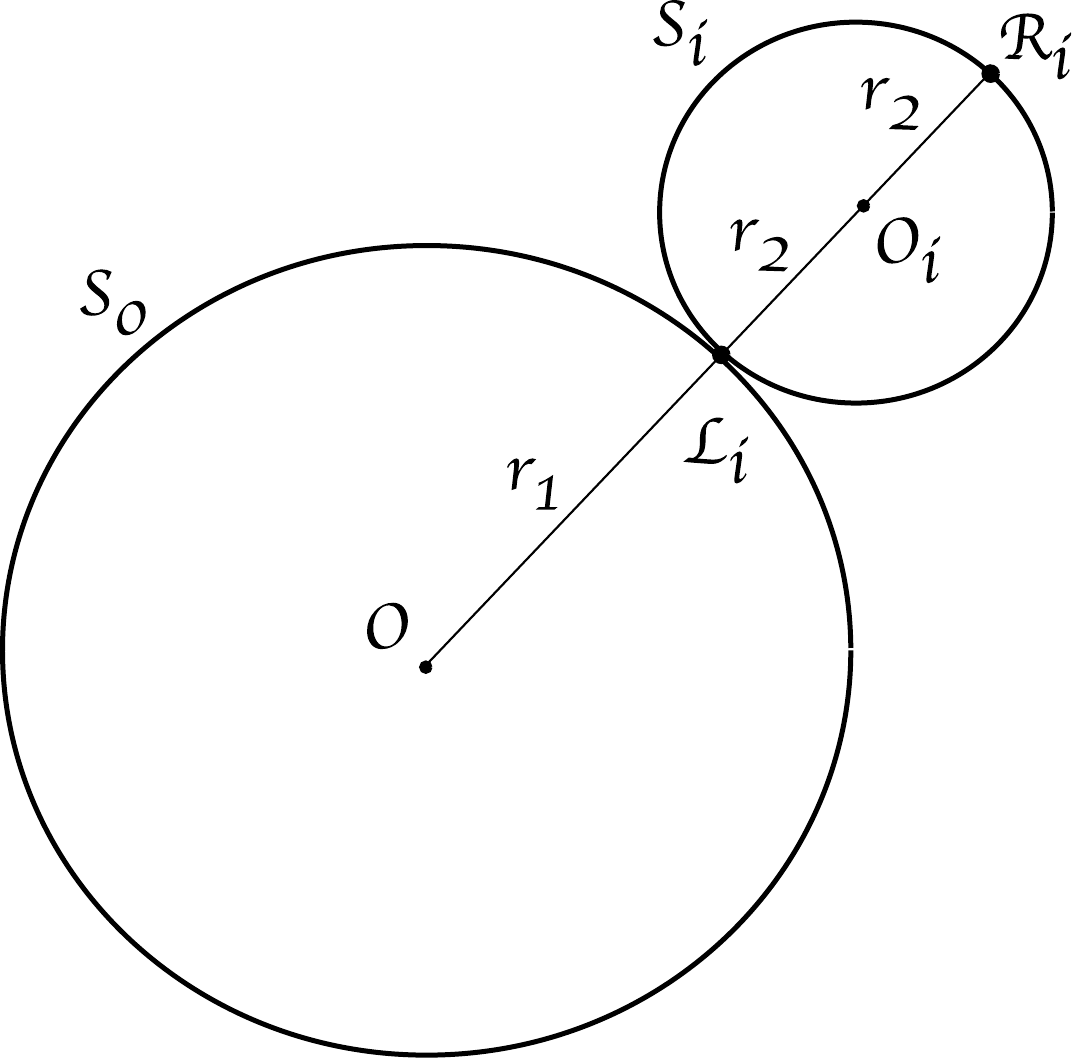}
\caption{A configuration of kissing hyperbolic spheres viewed as Euclidean ones.}\label{fig:spheres1}
\end{figure}

Indeed, the Euclidean distance from the centre $O$ of the ball model to any point at hyperbolic distance $r$ from $O$ equals $\tanh \frac{r}{2}$. If $O_i$ is the Euclidean centre\footnote{The hyperbolic centre of $S_i$ appears displaced towards the boundary of the ball model, as compared to $O_i$, and this is the reason for performing our computation of $r_1$ and $r_2$ below in a slightly peculiar way.} of any of the ``small'' spheres $S_i$, $i = 1, \dots, k$, then the hyperbolic distance between $O$ and the point of tangency $L_i$ between $S_0$ and $S_i$ is $r$, and it corresponds to the Euclidean radius of $S_0$, which equals $r_1$. The hyperbolic distance between $O$ and the point $R_i$ on the diameter of $S_i$ opposite to $L_i$ is $r + 2 r = 3 r$, while the corresponding Euclidean distance is $r_1 + 2 r_2$. Thus, we obtain $r_1 = \tanh \frac{r}{2}$, and $r_1 + 2 r_2 = \tanh\frac{3r}{2}$, from which the above values of $r_1$ and $r_2$ can be found. 

Choose one of the small spheres $S_i$, $i = 1, \dots, k$, with Euclidean centre $O_i$, and project it onto $S_0$ along the rays emanating from $O$: such a projection will be a spherical cap $C_i$ on $S_0$ of angular radius $\theta$ between a tangent line $OT_i$ to $S_i$ (with $T_i$ the point of tangency to $S_i$) and the line $OO_i$ connecting the centres of $S_0$ and $S_i$ (so that $OO_iT_i$ is a right triangle, with $\measuredangle OT_iO_i = \frac{\pi}{2}$). The area of such a cap is given by the formula 
\begin{equation*}
\mathrm{Area}\, C_i = \frac{1}{2}\cdot \mathrm{Area}\, S_0 \cdot \frac{B\left( \sin^2 \theta; \frac{n-1}{2}, \frac{1}{2} \right)}{B( \frac{n-1}{2}, \frac{1}{2} )}.
\end{equation*}
If we consider $S_0$ with the standard angular metric, then $\theta$ equals the radius of $C_i$, which can be interpreted as a spherical circle, i.e. the set of all points at angular distance $\theta$ from the centre $O'_i = OO_i \cap S_0$.

Since the points $O$, $O_i$ and $T_i$ are situated in a plane, a standard trigonometric computation yields (see Fig.~\ref{fig:triangle1})
\begin{equation*}
\sin \theta = \frac{r_2}{r_1 + r_2},
\end{equation*}
where $r_1 = \tanh \frac{r}{2}$ and $r_2 = \frac{1}{2} \left(\tanh\frac{3r}{2} - \tanh\frac{r}{2}\right)$, as given above. By substituting all of these quantities and simplifying the resulting formula, we obtain
\begin{equation*}
\sin \theta = \frac{\sech r}{2}. 
\end{equation*}

\begin{figure}[h]
\centering
\includegraphics[scale=0.45]{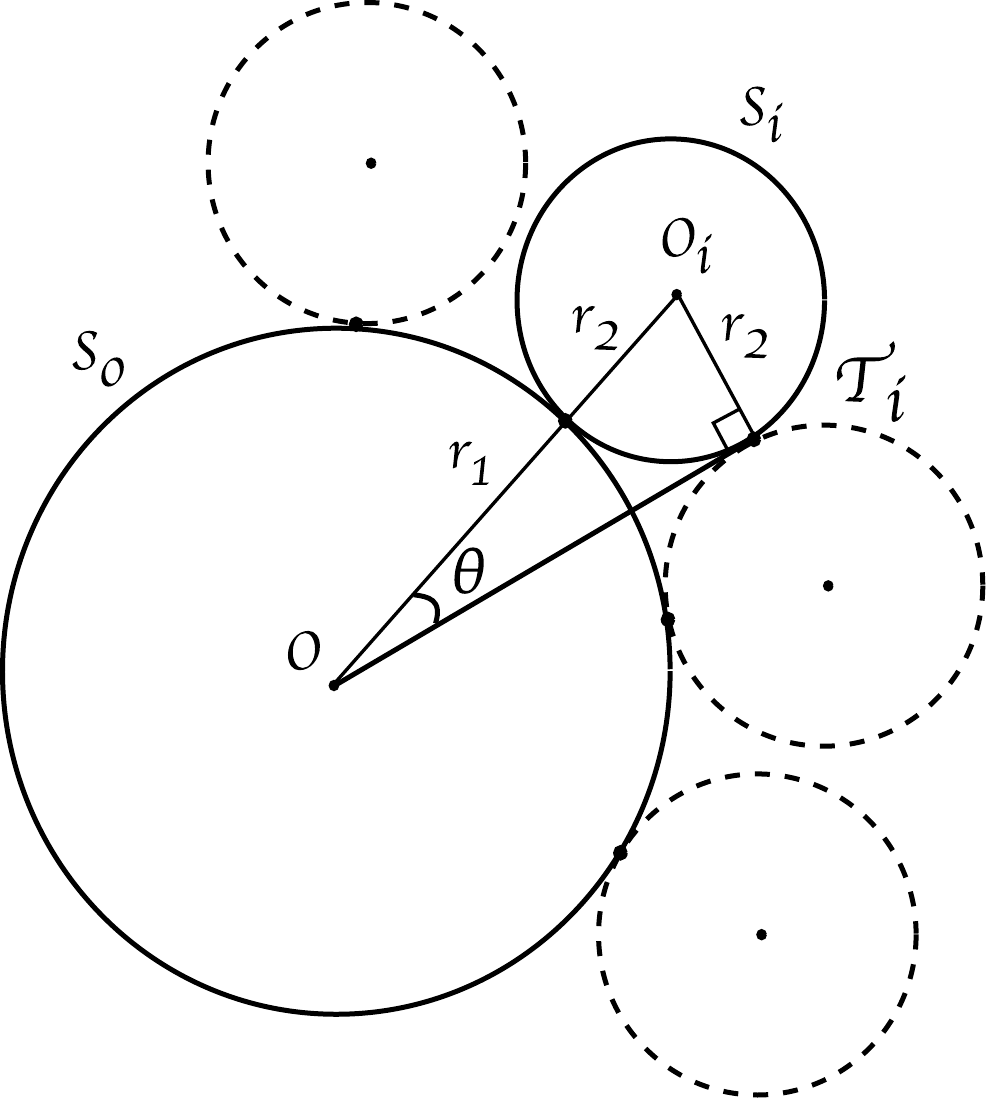}
\caption{The section of a kissing configuration by the $OO_iT_i$ plane.}\label{fig:triangle1} 
\end{figure}

Since the caps $C_i$ are mutually congruent and they pack the sphere $S_0$, we have that 
\begin{equation*}
\kappa_H(n, r)\cdot \mathrm{Area}\, C_i = \sum^k_{i=1} \mathrm{Area}\, C_i \leq \mathrm{Area}\, S_0,
\end{equation*}
and the theorem follows.
\end{proof}

\subsection{Lower bound}

Before providing a lower bound for $\kappa_H(n, r)$, let us first formulate two simple but crucial observations. The first one concerns a relation between packing and covering of the $n$-dimensional unit sphere $\mathbb{S}^n$ by closed metric balls. A packing of $\mathbb{S}^n$ by closed metric balls of angular radius $r>0$ with non-intersecting interiors is called \textit{maximal} if it cannot be enlarged by adding more such balls without overlapping their interiors.  

\begin{lemma}\label{lemma:pack-to-cover}
Let $\mathbb{S}^n$ be packed by closed metric balls $B_i$, $i=1, 2, \dots$, of equal (angular) radius $r$, and let such packing be maximal. Then $\mathbb{S}^n$ is covered by closed metric balls $B^\prime_i$, $i = 1, 2, \dots$, concentric to $B_i$, of radius $2 r$.
\end{lemma} 
\begin{proof}
Let $x_i$ be the centre of $B_i$, $i=1, 2, \dots$. For $x\in X$ let $\rho(x)$ be the shortest angular distance from $x$ to the set $\cup_i B_i$. Since $B_i$'s form a maximal packing, we have that $\rho(x) < r$. Thus, $x$ is at distance $\rho(x)$ from some $B_j$, which implies that $x$ is at distance $\rho(x) + r < 2r$ from its centre $x_j$. The latter means that $x \in B^\prime_j$.
\end{proof}

\begin{lemma}\label{lemma:caps-max-packing}
The packing of $S_0$ by the spherical caps $C_i$, $i = 1, 2, \dots, k$, from Theorem~\ref{thm:upper-bound-h}  is maximal, if $k = \kappa_H(n, r)$.
\end{lemma}
\begin{proof}
Let us consider two small spheres $S_i$ and $S_j$ tangent to $S_0$, and let $C_i$ and $C_j$ be the corresponding caps on $S_0$. Let $P$ be a two-dimensional plane through $O$, $O_i$, and $O_j$. Let $s_0 = S_0 \cap P$ be a circle of radius $r_1$ in the plane $P$, while $s_i = S_i \cap P$ and $s_j = S_j \cap P$ be the two outer circles of radius $r_2$ tangent to $s_0$. Let also $c_i = C_i\cap P$ and $c_j = C_j\cap P$. Observe that $s_i$ and $s_j$ intersect if and only if $c_i$ and $c_j$ intersect, for all $1 \leq i, j \leq k$. This is equivalent to $S_i$ intersecting $S_j$ if and only if $C_i$ intersects $C_j$.

After placing an additional cap $C_{k+1}$ (congruent to any of the already existing $C_i$'s) on $S_0$, we can create a sphere $S_{k+1}$ (congruent to any of $S_i$'s) producing $C_i$ as its central projection onto $S_0$.  By assumption, $k = \kappa_H(n, r)$, and thus there exists $S_l$ such that $S_{k+1}$ and $S_l$ intersect. Then, $C_{k+1}$ and $C_l$ also intersect, and thus the packing of $S_0$ by $C_i$'s is indeed maximal.
\end{proof}

Now we can formulate and prove a lower bound for $\kappa_H(n, r)$. 

\begin{thm}\label{thm:lower-bound-h}
For any integer $n \geq 2$ and a non-negative number $r\geq 0$, we have that 
\begin{equation*}
\kappa_H(n, r) \geq \frac{2\, B\left( \frac{n-1}{2}, \frac{1}{2} \right)}{B\left(\sech^2 r - \frac{\sech^4 r}{4}; \frac{n-1}{2}, \frac{1}{2} \right)},
\end{equation*} 
where $B(x; y, z) = \int^x_0 t^{y-1} (1-t)^{z-1} dt$, for all $x \in [0,1]$ and $y, z > 0$, is the incomplete beta-function, and $B(y, z) = B(1; y, z)$.
\end{thm}
\begin{proof}
By Lemma \ref{lemma:caps-max-packing}, the packing of $S_0$ by $k = \kappa_H(n, r)$ spherical caps $C_i$ produced in the proof of Theorem~\ref{thm:upper-bound-h} is maximal. Let $C^\prime_i$ be a spherical cap concentric to $C_i$ of angular radius $2\theta$. Since $X = S_0$ with angular metric on it can be thought of as a rescaled unit sphere $\mathbb{S}^{n-1}$, in which each $C_i$ is a closed metric ball of radius $\theta$, and each $C^\prime_i$ is a closed metric ball of radius $2 \theta$, we obtain that $C^\prime_i$'s cover $S_0$ by Lemma~\ref{lemma:pack-to-cover}. 

Then,
\begin{equation*}
\kappa_H(n, r)\cdot \mathrm{Area}\, C^\prime_i = \sum^k_{i=1} \mathrm{Area}\, C^\prime_i \geq \mathrm{Area}\, S_0,
\end{equation*}
where 
\begin{equation*}
\mathrm{Area}\, C^\prime_i = \frac{1}{2}\cdot \mathrm{Area}\, S_0 \cdot \frac{B\left( \sin^2(2 \theta); \frac{n-1}{2}, \frac{1}{2} \right)}{B( \frac{n-1}{2}, \frac{1}{2} )},
\end{equation*}
with $\theta$ such that
\begin{equation*}
\sin \theta = \frac{\sech r}{2}.
\end{equation*}
By using the formula $\sin(2 \theta) = 2 \sin \theta \cos \theta$, one can express $\sin^2(2\theta)$ through $\sin^2\theta$, and the theorem follows.
\end{proof}

\subsection{Euclidean kissing numbers}

The estimates of Theorems~\ref{thm:upper-bound-h} and \ref{thm:lower-bound-h} can be easily adapted to the case of the Euclidean kissing number, i.e. the kissing number of spheres of unit radii in $\mathbb{R}^n$, for $n\geq 2$. To this end, one needs just to set $\theta = \frac{\pi}{3}$ and $r = 0$ in the preceding argument.

\begin{cor}\label{cor:euclidean-k}
For the kissing number $\kappa(n)$ of unit radius spheres in $\mathbb{R}^n$, we have that
\begin{equation*}
\frac{2\, B\left( \frac{n-1}{2}, \frac{1}{2} \right)}{B\left( \frac{3}{4}; \frac{n-1}{2}, \frac{1}{2} \right)} \leq \kappa(n) \leq \frac{2\, B\left( \frac{n-1}{2}, \frac{1}{2} \right)}{B\left( \frac{1}{4}; \frac{n-1}{2}, \frac{1}{2} \right)}.
\end{equation*}
\end{cor}

The upper bound in the above theorem is identical to the one obtained by Glazyrin \cite[Theorem~6]{Glazyrin} albeit in a different context. Namely, the work \cite{Glazyrin} studies the average degree of contact graphs for kissing spheres having arbitrary radii. It is worth mentioning that in this case we shall have exactly the same estimates for hyperbolic and Euclidean spaces, since by allowing varying radii we effectively avoid the ambient metric influencing the combinatorics of kissing configurations\footnote{This is true since spheres in $\mathbb{H}^n$ (given by Poincar\'e's ball model) can be also described as spheres in $\mathbb{R}^n$ (i.e. by quadratic equations), and vice versa.}. 

Also, some asymptotic behaviour for the Euclidean kissing number $\kappa(n)$, as $n\rightarrow \infty$, can be deduced from Corollary~\ref{cor:euclidean-k}. 
However, the upper asymptotic bound in this case will be much poorer than $2^{0.401414(n+o(1))}$ by Kabatiansky and Levenshtein \cite{KL}. The lower asymptotic bound will be identical to the classical one by Chabauty \cite{C}, Shannon \cite{Shannon}, and Wyner \cite[\S 5.1]{W}, and is thus poorer than the recent bound by Jenssen, Joos, and Perkins \cite{JJP}. 

\subsection{Asymptotic behaviour}

First of all, we establish the following fact, which can be expressed simply by saying that the kissing number $\kappa_H(n, r)$ in $\mathbb{H}^n$, for any fixed dimension $n \geq 2$, grows exponentially fast with the radius $r$. Note, that the exponential asymptotic behaviour of the kissing number $\kappa_H(2, r)$, as $r \rightarrow \infty$, follows readily from the work by Bowen \cite{B}. 

\begin{cor}\label{corollary:asymptotic-k-H}
As $r \to \infty$, the following asymptotic formula holds:
\begin{equation*}
\kappa_H(n, r) \sim (n-1) \, d_{n-1} \, B\left( \frac{n-1}{2}, \frac{1}{2} \right)\, e^{(n-1) r},
\end{equation*}
where $d_n$ is the sphere packing density in $\mathbb{R}^n$.
\end{cor}

\begin{proof}
Let $d_n$ be the best packing density of $\mathbb{R}^n$ by unit spheres, and let $\delta_n(\theta)$ be the best packing density of $\mathbb{S}^n$ by spherical caps of (spherical) radius $\theta$. The quantity $\rho = 2 \sin\frac{\theta}{2}$ can be regraded as the ``Euclidean'' radius of such a spherical cap. Let $\overline{\delta}_n(\rho)$ be the best packing density of $\mathbb{S}^n$ by spherical caps with Euclidean radii $\rho$. 

As mentioned by Yaglom \cite{Yaglom}\footnote{a complete and rigorous proof was given by Levenshtein in \cite{Levenshtein-2}.}, 
\begin{equation*}
\lim_{\theta \to 0} \delta_n(\theta) = \lim_{\rho \to 0} \overline{\delta}_n(\rho) =  \lim_{\rho \to 0} (\text{packing density of } \frac{1}{\rho} \mathbb{S}^n \text{ by spherical caps of Euclidean radius } 1),
\end{equation*}
where $\frac{1}{\rho} \mathbb{S}^n$ is the rescaled unit sphere 
$\mathbb{S}^n$ with scaling factor $\frac{1}{\rho} \to \infty$. Since $\frac{1}{\rho}$ approximates $\mathbb{R}^n$ as $\rho \to 0$, we have that 
\begin{equation*}
\lim_{\theta \to 0} \delta_n(\theta) = \lim_{\rho \to 0} \overline{\delta}_n(\rho) = d_n. 
\end{equation*}

Thus, we have 
\begin{equation*}
\delta_{n-1}(\theta) = \sup_{C_1, \ldots, C_m} \frac{\sum^m_{i=1} \mathrm{Area}\, C_i}{\mathrm{Area}\, \mathbb{S}^{n-1}}, 
\end{equation*}
where $C_1, \ldots, C_m$ is a packing of radius $\theta$ caps on $\mathbb{S}^{n-1}$. Let us note, that the above can be written as
\begin{equation*}
\delta_{n-1}(\theta) = \kappa_H(n, r) \cdot  \frac{\mathrm{Area}\, C(\theta)}{\mathrm{Area}\, \mathbb{S}^{n-1}}, 
\end{equation*}
where $\sin \theta = \frac{\sech r}{2}$ and $\kappa_H(n ,r)$ is the kissing number for radius $r$ congruent spheres in $\mathbb{H}^n$, and $C(\theta)$ is a spherical cap of radius $\theta$. 

From the discussion above, it follows that
\begin{equation*}
\kappa_H(n, r) = \delta_{n-1}(\theta) \cdot \frac{2 B\left( \frac{n-1}{2}, \frac{1}{2} \right)}{B\left( \frac{\sech^2 r}{4}; \frac{n-1}{2}, \frac{1}{2} \right)}.
\end{equation*}

By using the series expansion of the integrand in the definition of $B(x; \frac{n-1}{2}, \frac{1}{2})$ around $x=0$ and integrating term-wise, then inverting the series, we obtain:
\begin{equation*}
\frac{1}{B(x; \frac{n-1}{2}, \frac{1}{2})} = \frac{n-1}{2} \cdot x^{\frac{1-n}{2}} \cdot \left(  1 + O(x) \right).
\end{equation*}

Since $\sech r \sim 2 e^{-r}$, as $r \to \infty$, and also $\lim_{\theta \to 0} \delta_{n-1}(\theta) = d_{n-1}$, we can combine all the above into the corollary's statement.
\end{proof}

In Section \ref{section:upper-lower-bounds}, by using SDP and by constructing concrete configurations of kissing spheres, respectively, we produce upper and lower bounds for $\kappa_H(n, r)$ in dimensions $n = 3, 4$ and for various radii. In dimension $n=3$, most of the bounds appear to be optimal. 

\section{Estimating the kissing number in spherical space}\label{theoreticalbound-S}

In this section we produce some upper and lower bounds for the kissing number $\kappa_S(n, r)$ in $\mathbb{S}^n$, $n\geq 2$, so that we can analyse its general behaviour, and some limiting values when the radius $r$ approaches $0$ or $\frac{\pi}{3}$. In this respect, we use classical geometric approach. We refer the reader to \cite[\S 2.1]{Ratcliffe} for the necessary basics of spherical geometry.

\subsection{Upper bound}

First we prove the following upper bound for $\kappa_S(n, r)$, with $0 \leq r \leq \frac{\pi}{3}$.

\begin{thm}\label{thm:upper-bound-s}
For any integer $n \geq 2$ and a non-negative number $r \leq \frac{\pi}{3}$, we have that 
\begin{equation*}
\kappa_S(n, r) \leq \frac{2\, B\left( \frac{n-1}{2}, \frac{1}{2} \right)}{B\left( \frac{\sec^2 r}{4}; \frac{n-1}{2}, \frac{1}{2} \right)},
\end{equation*} 
where $B(x; y, z) = \int^x_0 t^{y-1} (1-t)^{z-1} dt$, for all $x \in [0,1]$ and $y, z > 0$, is the incomplete beta-function, and $B(y, z) = B(1; y, z)$.
\end{thm}
\begin{proof}
Let $S_0$ be a radius $r$ sphere in $\mathbb{S}^n$ with centre $O$, and let $S_i$, $i = 1, \dots, k$, be its neighbours in a kissing configuration that are also radius $r$ spheres with centres $O_i$, $i = 1, \dots, k$, respectively. Consider a configuration of two tangent spheres: $S_0$ and $S_i$ for some $i$, $1\leq i \leq k$. Let $OT_i$ be the geodesic ray emanating from $O$ that is tangent to $S_i$ at point $T_i$. Let also $L_i$ be the point of tangency between $S_0$ and $S_i$, while $N_i$ is the intersection point of $S_0$ with $OT_i$.

Then in the $OO_iT_i$ spherical plane we have a right spherical triangle with vertices exactly $O$, $O_i$ and $T_i$. Let $\theta$ be the angle at $O$. Then by the spherical law of sines \cite[Theorem 2.5.2]{Ratcliffe}, we obtain 
\begin{equation*}
\sin |O_iT_i| = \sin |OO_i| \sin \theta,
\end{equation*}
which implies, once we substitute the lengths $|O_i T_i| = r$ and $|OO_i| = 2r$,
\begin{equation*}
\sin \theta = \frac{\sec r}{2}.
\end{equation*}
The condition $r < \frac{\pi}{3}$ ensures that the triangle $OO_iT_i$ does not degenerate into a spherical geodesic ``lune'' of angle $\theta = \frac{\pi}{2}$.  

If we project $S_i$ onto $S_0$ along the geodesic rays emanating from the centre $O$ of $S_0$, we obtain a ``cap'' $C_i$ on $S_0$, and all such caps resulting from a kissing configuration have non-intersecting interiors. However, we shall consider a purely Euclidean picture instead that can be obtained as follows. 

Since $S_0$ is a section of $\mathbb{S}^n \subset \mathbb{R}^{n+1}$ by a hyperplane $P_0$, let us consider the orthogonal projection $p$ of $\mathbb{R}^{n+1}$ onto $P_0$. Then the centre $O$ of $S_0$ is projected down to a point $O^* = p(O)$ in $P_0$, while we have $L^*_i = p(L_i) = L_i$, $N^*_i = p(N_i) = N_i$. Thus, the cap $C_i$ project to a cap $C^*_i = p(C_i) = C_i$ on the sphere $S^*_0 = p(S_0) = S_0$ in the hyperplane $P_0$. The cap $C^*_i$ has angular radius $\theta$ as measured on the surface of $S^*_0$, for all $1 \leq i \leq k$. 

Thus we obtain that
\begin{equation*}
\mathrm{Area}\, S^*_0 \geq \sum^{k}_{i=1} \mathrm{Area}\, C^*_i = \kappa_S(n,r) \cdot \mathrm{Area}\, C^*_i,
\end{equation*} 
where 
\begin{equation*}
\mathrm{Area}\, C^*_i = \frac{1}{2}\cdot \mathrm{Area}\, S^*_0 \cdot \frac{B\left( \sin^2 \theta; \frac{n-1}{2}, \frac{1}{2} \right)}{B( \frac{n-1}{2}, \frac{1}{2} )}.
\end{equation*}
We remark that the above formula is valid only for spherical caps of angular radius $\theta \leq \frac{\pi}{2}$. Otherwise, the resulting area will be that of the complementary region to $C^*_i$ in $S^*_0$. As our condition $r \leq \frac{\pi}{3}$ implies $\theta \leq \frac{\pi}{2}$, the theorem follows.  
\end{proof}

\subsection{Lower bound}

In order to obtain a lower bound, we shall again use the interplay between packing and covering provided by Lemma~\ref{lemma:pack-to-cover}. 

\begin{thm}\label{thm:lower-bound-s}
For any integer $n \geq 2$ and a non-negative number $r$, we have that 
\begin{equation*}
\kappa_S(n, r) \geq \left\{ 
\begin{array}{cc}
\frac{2\, B\left( \frac{n-1}{2}, \frac{1}{2} \right)}{B\left(\sec^2 r - \frac{\sec^4 r}{4}; \frac{n-1}{2}, \frac{1}{2} \right)}, &\text{ if } 0 \leq r \leq \frac{\pi}{4}, \\
\frac{2\, B\left( \frac{n-1}{2}, \frac{1}{2} \right)}{2\, B\left( \frac{n-1}{2}, \frac{1}{2} \right) - B\left(\sec^2 r - \frac{\sec^4 r}{4}; \frac{n-1}{2}, \frac{1}{2} \right)}, &\text{ if } \frac{\pi}{4} \leq r \leq \frac{\pi}{3}.
\end{array}  
\right. 
\end{equation*} 
where $B(x; y, z) = \int^x_0 t^{y-1} (1-t)^{z-1} dt$, for all $x \in [0,1]$ and $y, z > 0$, is the incomplete beta-function, and $B(y, z) = B(1; y, z)$.
\end{thm}
\begin{proof}
Let us observe that the packing of $S^*_0$ by the spherical caps $C^*_i$, $i = 1, 2, \dots, k$, from Theorem~\ref{thm:upper-bound-s}  is maximal, if $k = \kappa_H(n, r)$. Since rescaling does not change angular distances, we may assume that $S^*_0$ has unit radius. 

Let $C^\prime_i$ be a spherical cap concentric to $C^*_i$ of angular radius $2\theta$. By Lemma~\ref{lemma:pack-to-cover},  $C^\prime_i$'s cover $S_0$. Then 
\begin{equation*}
\kappa_S(n, r)\cdot \mathrm{Area}\, C^\prime_i = \sum^k_{i=1} \mathrm{Area}\, C^\prime_i \geq \mathrm{Area}\, S^*_0,
\end{equation*}
where 
\begin{equation*}
\mathrm{Area}\, C^\prime_i =   \frac{1}{2}\cdot \mathrm{Area}\, S^*_0 \cdot \left\{ \begin{array}{cc}
\frac{B\left(\sin^2(2 \theta),\frac{n-1}{2}, \frac{1}{2} \right)}{B( \frac{n-1}{2}, \frac{1}{2} )}, &\text{ if } 0\leq \theta \leq \frac{\pi}{4},\\
2 - \frac{B\left(\sin^2(2 \theta); \frac{n-1}{2}, \frac{1}{2} \right)}{B( \frac{n-1}{2}, \frac{1}{2} )}, &\text{ if } \frac{\pi}{4} \leq \theta \leq \frac{\pi}{2}. \end{array} \right.
\end{equation*}
depending on whether $\theta \leq \frac{\pi}{4}$ (then the first value realises the cap area) or $\frac{\pi}{4} \leq \theta \leq \frac{\pi}{2}$ (then the second value realises the cap area, while the first one gives the complementary region area), and with $\theta$ satisfying $\sin \theta = \frac{\sec r}{2}$, as before.

By using the formula $\sin(2 \theta) = 2 \sin \theta \cos \theta$, the theorem follows after a straightforward computation.
\end{proof}

The above bound comes from an argument completely analogous to that by Wyner \cite{W}. According to the recent results by Jenssen, Joos, and Perkins, it can be improved by a linear factor in $n$ provided $0 \leq r \leq \frac{\pi}{4}$, c.f. \cite[Theorem 2]{JJP}.

\subsection{Limiting values of kissing numbers}

By putting $r=0$ in the formulas of Theorem \ref{thm:upper-bound-s} and Theorem \ref{thm:lower-bound-s}, we obtain that $\sin \theta = \frac{1}{2}$, or $\theta = \frac{\pi}{3}$, which produces the usual (and rather imprecise) estimates for the Euclidean kissing number $\kappa(n)$. 

As noticed in Section~\ref{sdpbound}, the function $\kappa_S(n, r)$ is a decreasing function in $r$, $0 \leq r \leq \pi/3$, for any $n\geq 2$. Clearly, $\kappa_S(n, r) \leq \kappa(n)$. Thus, the limit $\lim_{r\to 0} \kappa_S(n, r)$ exists, though it might not be equal to $\kappa(n)$. Indeed, $\lim_{r\to 0} \kappa_S(2, r) = 5$, while $\kappa(2) = 6$. On the other hand, $\lim_{r \to 0} \kappa_S(3, r) = \kappa(3) = 12$. Conjecturally, $\lim_{r\to 0} \kappa_S(4, r) = 22$, while it is known that $\kappa(4) = 24$, c.f. \cite{Musin, Musin2}. 

Another extreme end is $\kappa_S(n, \pi/3) = 2$, for all $n\geq 2$. First of all, we obtain from Theorem \ref{thm:upper-bound-s} and Theorem \ref{thm:lower-bound-s} that $1 \leq \kappa_S(n, \pi/3) \leq 2$. Now, let us consider the points $a = (0,-1,0,\dots,0)$, $b = (\sqrt{3}/2,1/2,0,\dots,0)$ and $c=(-\sqrt{3}/2,1/2,0,\dots,0)$ in $\mathbb{S}^n$. Notice that $a$, $b$ and $c$ are placed at mutually equal distances of $2\pi/3$. Thus, the spheres $S_a$, $S_b$ and $S_c$ of radii $\pi/3$, centred at the respective points, are mutually tangent. Each of them has two congruent neighbours and thus $\kappa_S(n,\pi/3) \geq 2$.

It is also clear from the upper bound in Theorem~\ref{thm:upper-bound-s} that $\kappa_S(n, r) = 1$ for $\frac{\pi}{3} < r \leq \frac{\pi}{2}$ and $n \geq 1$. As $r = \frac{\pi}{2}$, the sphere of radius $r$ fills a hemisphere of $\mathbb{S}^n$, and thus $\kappa_S(n, r) = 0$ for $\frac{\pi}{2} < r \leq \pi$ for $n \geq 1$.

In Section~\ref{section:upper-lower-bounds}, by using SDP, by constructing concrete configurations of kissing spheres, and by using the known solutions to Tammes' problem \cite{Danzer, FejesToth, SchuetteWaerden}, we approximate $\kappa_S(n, r)$ as a step function for dimensions $n = 3, 4$ and radii $0 \leq r \leq \pi$.

\section{The semidefinite programming bound}\label{sdpbound}
For $x, y \in \mathbb{R}^n$, we denote by $x \cdot y = x_1 y_1 + x_2 y_2 + \ldots + x_n y_n$ their Euclidean inner product, and let $\mathbb{S}^{n-1} = \{ x \in \mathbb{R}^n : x \cdot x = 1\}$ be the $(n-1)$-dimensional unit sphere. Furthermore, let the \textit{angular distance} between $x, y \in \mathbb{S}^{n-1}$ be $d(x, y) = \arccos(x \cdot y)$, c.f. \cite[\S 2.1]{Ratcliffe}. 

In order to determine upper bounds for the kissing number, we first consider the more general problem of finding the maximal number of points on the unit sphere with minimal angular distance $\theta$. This problem is defined by
$$A(n, \theta) = \max\left\{|C| : C \subset \mathbb{S}^{n-1} \text{ and } x \cdot y \leq \cos \theta\text{ for all distinct }x, y \in C \right\}.$$
Note that the Euclidean kissing number equals $\kappa(n) = A(n, \pi/3)$. 

A set of points  $C \subset \mathbb{S}^{n-1}$ with $d(x, y) \geq \theta$ for all distinct $x, y \in C$ is called a \textit{spherical code with minimal angular distance} $\theta$. Then the kissing number $\kappa_H(n, r)$ of radius $r$ spheres in $\mathbb{H}^n$ is equal to the cardinality of a maximal spherical code $C \in \mathbb{S}^{n-1}$ with $x\cdot y \leq 1 - \frac{1}{1+ \cosh(2r)}$ for $x, y \in C$. 

\begin{lemma}\label{kappanr-h}
\begin{equation*}
\kappa_H(n, r) =  \max\left\{|C| : C \subset \mathbb{S}^{n-1} \text{ and } x\cdot y \leq 1 - \frac{1}{1+ \cosh(2r)} \text{ for all distinct } x, y \in C \right\}.
\end{equation*}
\end{lemma}

\begin{proof}
Consider an equilateral hyperbolic triangle with side length $2r$, and let $\theta$ be one of its inner angles. By the hyperbolic law of cosines, we obtain
$$\cos \theta  = \frac{\cosh^2 (2r) - \cosh (2r)}{\sinh^2 (2r)} = \frac{\cosh (2r) (\cosh (2r) -1)}{\cosh^2 (2r) -1} = \frac{\cosh (2r)}{1 + \cosh (2r)} = 1 - \frac{1}{1+ \cosh(2r)}.$$
\end{proof}

From Lemma~\ref{kappanr-h} it becomes clear that $\kappa_H(n, r)$ is an increasing function of $r\geq 0$, for all $n\geq 1$, since $\frac{1}{1+ \cosh(2r)}$ is decreasing with $r$, and thus the set of possible codes $C$ becomes larger as $r$ increases. Analogously, $\kappa_H(n+1, r) \geq \kappa_H(n, r)$, for all $n\geq 1$ and $r\geq 0$, since $\mathbb{S}^{n-1} \subset \mathbb{S}^n$, as $n$ increases.

The optimal value of the semidefinite program of Bachoc and Vallentin \cite{BachocVallentin} is an upper bound of $A(n, \theta)$. We can adapt their program to obtain upper bounds for the kissing number in hyperbolic space.

For $n \geq 3$, let $P^n_k(u)$ denote the Jacobi polynomial of degree $k$ and parameters $((n-3)/2, (n-3)/2)$, normalized by $P^n_k(1) = 1$. If $n = 2$, then $P^n_k(u)$ denotes the Chebyshev polynomial of the first kind of degree $k$. For a fixed integer $d > 0$, we define $Y^n_k$ to be a $(d-k+1) \times (d-k+1)$ matrix whose entries are polynomials on the variables $u, v, t$ defined by
$$(Y^n_k)_{i, j}(u, v, t) = P^{n+2k}_i(u) P^{n+2k}_j(v) Q^{n-1}_i(u, v, t),$$
for $ 0 \leq i, j \leq d-k$, where
$$Q^{n-1}_k(u, v, t) = ((1-u^2)(1-v^2))^{k/2} P^{n-1}_k\left(\frac{t-uv}{\sqrt{(1-u^2)(1-v^2)}}\right).$$
The symmetric group on three elements $\mathcal{S}_3$ acts on a triple $(u, v, t)$ by permuting its components. This induces the action
$\sigma p(u, v, t) = p(\sigma^{-1}(u, v, t))$
on $\mathbb{R}[u, v, t]$, where $\sigma \in \mathcal{S}_3$. By taking the group average of $Y^n_k$, we obtain the matrix
$$S^n_k(u,v,t) = \frac{1}{6} \sum_{\sigma \in \mathcal{S}_3} \sigma Y^n_k(u,v,t), $$
whose entries are invariant under the action of $\mathcal{S}_3$.

A symmetric matrix $A \in \mathbb{R}^{n\times n}$ is called positive semidefinite if all its eigenvalues are non-negative. We write this as $A \succeq 0$. A \textit{semidefinite program} (SDP) is an optimisation problem for a linear function over a set of positive semidefinite matrices restricted by linear matrix equalities. For $A, B \in \mathbb{R}^{n\times n}$, let $\langle A,B \rangle = \text{tr}(B^TA)$ be the trace inner product. Also, we define
$$\bigtriangleup = \{(u, v, t) \in \mathbb{R}^3 : -1 \leq u \leq v \leq t \leq \cos \theta \text{ and } 1+2uvt-u^2-v^2-t^2 \geq 0\}$$
and
$$\bigtriangleup_0 = \{ (u, u, 1) : -1 \leq u \leq \cos \theta\}.$$
The triples $(u,v,t) \in \bigtriangleup$ are possible inner products between three points in a spherical code in $\mathbb{S}^{n-1}$ with minimal angular distance $\theta$. Hence, 
\begin{align*}
(u,v,t) \in \bigtriangleup \text{ if and only if } \exists\,\, x, y, z \in \mathbb{S}^{n-1}:\,\, & x \cdot y \leq \cos \theta,\, x \cdot z \leq \cos \theta,\, y \cdot z  \leq \cos \theta, \\
&x \cdot y = u,\, x \cdot z = v,\, y \cdot z = t.
\end{align*}  
In \cite{BachocVallentin}, Bachoc and Vallentin proved the following theorem, where $J$ denotes the ``all 1's'' matrix.
\begin{thm}\label{BVsdp} Any feasible solution of the following optimisation program gives an upper bound on $A(n, \theta)$:
\begin{align*}
 \min~~ &1 + \sum_{k=1}^d a_k + b_{11} + \langle J, F_0\rangle,\\
 & a_k \geq 0 \text{ for } k=1, \ldots, d, \\
& \begin{pmatrix}
b_{11} & b_{12} \\ b_{21} & b_{22}
\end{pmatrix} \succeq 0 , \\
& F_k \in \mathbb{R}^{(d-k+1) \times (d-k+1)} \text{ and } F_k \succeq 0 \text{ for } k = 0, \ldots, d,\\
&(i)~ \sum_{k=1}^d a_k P^n_k(u) + 2b_{12} + b_{22} + 3\sum_{k=0}^d\langle S^n_k(u,u,1),F_k \rangle \leq -1 \text { for } (u,u,1) \in \bigtriangleup_0, \\
&(ii)~ b_{22} + \sum_{k=0}^d\langle S^n_k(u,v,t)F_k \rangle \leq 0 \text{ for } (u,v,t) \in \bigtriangleup.
\end{align*}
\end{thm}

The conditions $(i)$ and $(ii)$ of the previous program are polynomial constraints where we have to check that certain polynomials are non-negative in a given domain. Constraints of this kind can be written as sum-of-squares conditions. A polynomial $p$ is said to be a \textit{sum-of-squares polynomial} if and only if there exists polynomials $q_1, \ldots, q_m$ such that $$p = q_1^2 + \ldots + q_m^2.$$ To be a sum-of-squares polynomial is a sufficient condition for non-negativity. Analogously, one can also use sum-of-squares to check non-negativity for a certain domain: if there exists sum-of-squares polynomials $q_1,q_2$ such that
$$p(x) = q_1(x) + (b-x)(x-a)q_2(x)$$
then $p(x) \geq 0$ for $a \leq x \leq b$. Using sum-of-squares relaxations we can formulate the program in Theorem \ref{BVsdp} as an SDP. 

In order to obtain a finite-dimensional SDP which we can solve in practice, we have to fix the degree of the polynomials we consider for the sum-of-squares conditions. A polynomial $p(x_1, \ldots, x_n) \in \mathbb{R}[x_1, \ldots, x_n]$ of degree $2d$ can be written as a sum-of-squares if and only if there exists a positive semidefinite matrix $X$ such that $$p(x_1, \ldots, x_n) = \left\langle X,v^d(x_1, \ldots, x_n)v^d(x_1, \ldots, x_n)^T\right\rangle,$$ where $v^d(x_1, \ldots, x_n)\in \mathbb{R}[x_1, \ldots, x_n]^{\binom{n+d}{d}}$ is a vector which contains a basis of the space of real polynomials up to degree $d$. We denote $V^d(x) = v^d(x)v^d(x)^T$.

Hence, we obtain for the above conditions $(i)$ and $(ii)$ the following sum-of-squares relaxations
\begin{align*}
(i)~ \sum_{k=1}^d a_k P^n_k(u) + 2b_{12} + b_{22} + 3\sum_{k=0}^d\langle S^n_k(u,u,1)&, F_k \rangle + 1 +  \langle Q_0, V^{d}(u)\rangle \\
&+ (u + 1)(\cos \theta - u) \langle Q_1, V^{d-1}(u)\rangle = 0
\end{align*}
\begin{align*}
(ii)~ b_{22} &+ \sum_{k=0}^d\langle S^n_k(u,v,t), F_k \rangle +\langle R, V^{d}(u,v,t)\rangle + (u + 1)(\cos \theta - u) \langle R_0, V^{d-1}(u,v,t)\rangle \\
&+ (v + 1)(\cos \theta - v) \langle R_1, V^{d-1}(u,v,t)\rangle + (t + 1)(\cos \theta - t) \langle R_2, V^{d-1}(u,v,t)\rangle \\
&+ (1+2uvt-u^2-v^2-t^2) \langle R_3, V^{d-1}(u,v,t)\rangle = 0,
\end{align*}
where $Q_0, Q_1, R, R_0, \ldots, R_3$ are positive semidefinite matrices.

Bachoc and Vallentin proved that any feasible solution of the SDP in Theorem \ref{BVsdp} gives an upper bound on  $$ \max\left\{|C| : C \subset \mathbb{S}^{n-1} \text{ and } x\cdot y \leq \cos \theta\text{ for all distinct }x,y \in C \right\}.$$ Hence, by applying Lemma \ref{kappanr-h} we can replace $\cos \theta$ by $1 - \frac{1}{1+\cosh (2r)}$.

\begin{cor}\label{cor:bound-H}
Any feasible solution of the optimisation program in Theorem \ref{BVsdp} with 
$$\bigtriangleup = \left\{(u,v,t) \in \mathbb{R}^3 : -1 \leq u \leq v \leq t \leq 1 - \frac{1}{1+ \cosh(2r)} \text{ and } 1+2uvt-u^2-v^2-t^2 \geq 0\right\}$$
and
$$\bigtriangleup_0 = \left\{ (u,u,1) : -1 \leq u \leq 1 - \frac{1}{1+ \cosh(2r)}\right\}$$~gives an upper bound on $\kappa_H(n,r)$.
\end{cor}

In the case of the kissing number in $\mathbb{S}^n$, $n\geq 2$, an analogous approach can be used. First of all, we can describe $\kappa_S(n,r)$ in terms of spherical codes.

\begin{lemma}\label{kappanr-s}
\begin{equation*}
\kappa_S(n, r) =  \max\left\{|C| : C \subset \mathbb{S}^{n-1} \text{ and } x\cdot y \leq 1 - \frac{1}{1+ \cos(2r)} \text{ for all distinct } x, y \in C \right\}.
\end{equation*}
\end{lemma}

\begin{proof}
Consider a spherical triangle with side length $2r$, and let $\theta$ be one of its inner angles. By applying the spherical law of cosines \cite[Theorem 2.5.3]{Ratcliffe}, we obtain
$$\cos \theta  = \frac{\cos (2r) - \cos^2 (2r)}{\sin^2 (2r)} = \frac{\cos (2r) ( 1- \cos (2r))}{1 - \cos^2 (2r)} = \frac{\cos (2r)}{1 + \cos (2r)} = 1 - \frac{1}{1+ \cos(2r)}.$$
\end{proof}

From Lemma~\ref{kappanr-s} we immediately deduce that $\kappa_S(n, r)$ is a decreasing function in $0 \leq r \leq \frac{\pi}{3}$, for any $n\geq 1$, since $\frac{1}{1+ \cos(2r)}$ is increasing with $r$. Indeed, the set of possible spherical codes $C$ in the definition of $\kappa_S(n, r)$ from Lemma~\ref{kappanr-s} becomes smaller as $r$ increases. Thus $\kappa_S(n, r)$ is a decreasing step function of $r$, for any fixed dimension $n \geq 1$. Also, $\kappa_S(n+1, r) \geq \kappa_S(n, r)$, for $n\geq 1$, and $0\leq r \leq \frac{\pi}{3}$, since $\mathbb{S}^{n-1} \subset \mathbb{S}^{n}$, and the above argument applies in Lemma~\ref{kappanr-s} again.

\medskip
Then, an analogue of Corollary~\ref{cor:bound-H} can be formulated.

\begin{cor}\label{cor:bound-S} 
Any feasible solution of the optimisation program in Theorem \ref{BVsdp} with 
$$\bigtriangleup = \left\{(u,v,t) \in \mathbb{R}^3 : -1 \leq u \leq v \leq t \leq 1 - \frac{1}{1+ \cos(2r)} \text{ and } 1+2uvt-u^2-v^2-t^2 \geq 0\right\}$$
and
$$\bigtriangleup_0 = \left\{ (u,u,1) : -1 \leq u \leq 1 - \frac{1}{1+ \cos(2r)}\right\}$$
gives an upper bound on $\kappa_S(n,r)$.
\end{cor}

In \cite{Dostert}, Dostert, de Laat, and Moustrou published a library for computing exact semidefinite programming bounds for packing problems. This library provides the implementation of the semidefinite program given in Theorem \ref{BVsdp}. We obtain upper bounds for the kissing number in spherical and hyperbolic space by using the function \textsf{threepointsdp(n,d,costheta)} with costheta = $1- 1/(1+\cos(2r))$ or costheta = $1- 1/(1+\cosh(2r))$. Furthermore, using this library we obtain an exact solution from the  floating point solver solution and a rigorous feasibility check. Further information about the verification script can be found in the ancillary files available on GitHub\footnote{\url{https://github.com/sashakolpakov/non-euclidean-kissing-number} \label{github}}.

Having the upper bounds from the SDPs and lower bounds from concrete configurations in Section~\ref{configurations} for any given value of $0 \leq r \leq \frac{\pi}{3}$ that are sufficiently close to each other provides us with an approximate shape of $\kappa_S(n, r)$.  

\section{Kissing configurations from spherical codes}\label{configurations}

\subsection{Configurations in $\mathbb{H}^n$} A feasible kissing configuration in dimension $n\geq 2$ for radius $r > 0$ spheres is given by a spherical code $ C \in \mathbb{S}^{n-1}$ where  $x\cdot y \leq 1-\frac{1}{1+\cosh(2r)}\text{ for all distinct } x, y \in C.$ Any feasible spherical code $C$ gives a lower bound on the kissing number, hence $\kappa_H(n,r) \geq |C|$.

Since $1-\frac{1}{1+\cosh(2r)}$ increases with $r$ increasing, any kissing configuration for radius $r$ is also a kissing configuration for $\kappa_H(n,r')$ where $r' \geq r$.

Let $x_i \in \mathbb{R}^n$, for $i = 1, \dots, k$, be the approximate numeric coordinates of the code elements in a spherical code $C \subset \mathbb{S}^{n-1}$. Here we use \cite{Sloane-et-al} as a source of putatively optimal spherical codes on $\mathbb{S}^{n-1}$, for $n = 3, 4$. We define $\tilde{x}_i \in \mathbb{Q}^n$ to be a rational approximation of $x_i$ (in many cases, $x_i$ is a real number with $16$ digit precision, and thus is already approximated by a rational). After normalizing $\tilde{x}_i$ to norm $1$, we obtain that there exist $a_i \in \mathbb{Q}$, $b_i \in \mathbb{Q}^n$ such that $\tilde{x}_i = \sqrt{a_i}\cdot b_i$. 

Using interval arithmetic in \textsf{SageMath} \cite{sage}, we compute the maximal inner product of $\tilde{x}_i$ and $\tilde{x}_j$ for $i, j \in \{1, \ldots, k\}$, $i \neq j$. Let $r \in \mathbb{R}$ be such that the maximal inner product is at most $1-\frac{1}{1+\cosh(2r)}$. Since $\tilde{x}_i \in \mathbb{S}^{n-1}$ for all $i \in \{1, \ldots, k\}$, this exact spherical code (having exact values for its elements) defines a feasible kissing configuration of $k$ radius $r$ spheres in $\mathbb{H}^n$.  Note that while turning the approximate solution into an exact kissing configuration we might have to vary $r$ slightly. The \textsf{SageMath} code converting the approximate codes from \cite{Sloane-et-al} to their rationalised forms is available on GitHub\textsuperscript{\ref{github}}.

\subsection{Configurations in $\mathbb{S}^n$} Here we use exactly the same approach as for $\mathbb{H}^n$, with the only modification that we look for codes with the maximal inner product between codewords at most $1 - \frac{1}{1+\cos(2r)}$, for $r \in \mathbb{R}$. Since $\tilde{x}_i \in \mathbb{S}^{n-1}$ for all $i \in \{1, \ldots, k\}$, this exact spherical code (having exact values for its elements) defines a feasible kissing configuration of $k$ radius $r$ spheres in $\mathbb{H}^n$. The respective data is available on GitHub\textsuperscript{\ref{github}}.

\section{Upper and lower bounds}\label{section:upper-lower-bounds}

\subsection{Hyperbolic space} In this section, we provide concrete upper bounds for the kissing function in dimensions $n = 3$ and $4$ for certain radii $r$ by using several approaches: the SDP from Corollary~\ref{cor:bound-H}, and also the theoretical upper bounds due to  Levenshtein \cite{Levenshtein} and Coxeter \cite{Bor, Coxeter, FT}. The latter are much better than the geometric upper bounds in Theorem~\ref{thm:upper-bound-h}. For solving the SDP we used $d = 10$ or $d=12$. 

Furthermore, we compare the obtained results with the theoretical lower bounds that we get from Theorem \ref{thm:lower-bound-h}, as well as from the feasible configurations in Section \ref{configurations}. The computational results are given in Table~\ref{table:H3} and Table~\ref{table:H4}.

It follows from the proof of Theorem~\ref{thm:upper-bound-h} that by rounding the theoretical upper bound we shall obtain the number of circles in the kissing configuration in dimension $n=2$.  Another observation is that the theoretical upper bounds in dimension $n=3$ keep relatively close to the SDP upper bounds, while in dimension $4$ they diverge quite quickly. 

\begin{table}[ht]
\caption{Bounds for the kissing number in $\mathbb{H}^3$}
\begin{center}
\begin{tabular}{ l  c  c  c  c  c }
\hline \vspace*{-0.8em} \\
    & theoretical & lower bound  & SDP  & Levenshtein & Coxeter \\ 
  r & lower bound & by construction & upper bound & bound & bound \vspace*{0.1em}\\  \hline \vspace*{-0.8em} \\
0 & 4 &  12 \cite{SchuetteWaerden} & 12.368591\cite{Machado}  & 13.2857 & 13.3973\\
0.3007680932244 & 4.37289 & 13 & 13.66695 & 14.6365 & 14.7591\\
0.3741678937820 & 4.58663 & 14 & 14.57930 & 15.4829 & 15.5389\\
0.4603413898301 & 4.90925 & 15 & 15.76145& 16.6843 &  16.7150\\
0.5150988762761 & 5.15856 & 16 &  16.63748& 17.5619&  17.6233\\
0.5575414271933 & 5.37771 & 17 & 17.39631 & 18.3659 & 18.4214\\
0.6117193853329 & 5.69307 & 18 & 18.57836 & 19.5957 & 19.5694\\
0.6752402229782 & 6.1184 & 19 & 20.12475 & 21.1343 &  21.1170\\
0.6839781903772 & 6.18194 & 20 & 20.43374 & 21.3570 & 21.3482\\
0.7441766799717 & 6.65554 & 21 & 21.88751  & 23.0631 &23.0705\\
0.7727858684533 & 6.90384 & 22 & 22.81495 & 24.0041 & 23.9732\\
0.8064065300517 & 7.21623 & 23 & 24.08326 & 25.2137 & 25.1087\\
0.8070321648835 & 7.22226 & 24 & 24.32215 & 25.2348 & 25.1306
\end{tabular}
\end{center}
\label{table:H3}
\end{table}

\vspace*{0.1in}

\begin{table}[ht]
\caption{Bounds for the kissing number in $\mathbb{H}^4$}
\begin{center}
\begin{tabular}{ l  c  c  c  c  c }
\hline \vspace*{-0.8em} \\
    & theoretical & lower bound  & SDP  & Levenshtein & Coxeter \\ 
  r & lower bound & by construction & upper bound &  bound &  bound \vspace*{0.1em}\\  \hline \vspace*{-0.8em} \\
0	& 5.11506& 24 \cite{Musin} & 24.05691\cite{Machado}   & 26 & 26.4420\\
0.2803065634764      & 5.70802 & 25 & 28.36959 & 29.9154 &  29.9757\\
0.2937915284847 & 5.76935 & 26 &      28.54566 & 30.2755 & 30.3417\\
0.3533981811745      & 6.08306 & 27 & 30.35228 & 32.0432 & 32.2152\\
0.4029707622959 & 6.40115 & 29 &      32.37496 & 33.8969 &  34.1172\\
0.4361470369242 & 6.64597 & 30 &      33.73058 & 35.3805 & 35.5826\\
\end{tabular}
\end{center}
\label{table:H4}
\end{table}

\subsection{Spherical space} In Table \ref{table:S3}, we give the approximate values of $r$ corresponding to the``jumps'' for the kissing number $\kappa_S(3, r)$, by using the known maximal separation distances from the solutions of Tammes' problem \cite{Danzer, FejesToth, SchuetteWaerden}.  

In Table~\ref{table:S4}, the upper bounds from SDPs and lower bounds from concrete configurations for a few values of $r \in [0, \pi/3]$ provide us with an approximate shape of $\kappa_S(4, r)$. Here, we consider the spherical codes for $0 \leq r \leq \pi/3$ determined by the approach of Section \ref{configurations}. For each radius $r$ of these exact spherical codes, we compute the lower bound given by Theorem \ref{thm:lower-bound-s}, as well as the upper bound by the SDP from Corollary \ref{cor:bound-S} together with the ones due to Levenshtein \cite{Levenshtein} and Coxeter \cite{Bor, Coxeter, FT}. In order to obtain the given SDP upper bounds we use $d=8$.

\begin{table}[ht]
\caption{Jumps of $\kappa_S(3, r)$. Here we put $\kappa_S(3, 0) = \lim_{r\to 0} \kappa_S(3, r) =  12$.}
\begin{center}
\begin{tabular}{ c | c | c}
\hline \vspace*{-0.8em} \\
  Jump $n \to m$ & Exact value of $r$ & Approximate value of $r$\vspace*{0.1em}\\  \hline \vspace*{-0.8em} \\
$12 \to 10$ & $\pi / 10$ & 0.3141592653590 \\
$10 \to 9$  & $\arccos \tau$, with $16 \tau^6 - 44 \tau^4 + 34 \tau^2 - 7 = 0$ & 0.4122234203273 \\
$9 \to 8$ & $\pi / 6$ & 0.5235987755983 \\ 
$8 \to 7$ & $\frac{1}{2} \mathrm{arcsec}(2\sqrt{2})$ & 0.6047146014441 \\
$7 \to 6$ & $\frac{\pi}{4} + \frac{1}{2} \mathrm{arccsc}\left(2 - \csc \frac{\pi}{18}\right)$ & 0.6507545374483 \\
$6 \to 4$ & $\pi / 4$ & 0.7853981633974 \\
$4 \to 3$ & $\frac{1}{2} \arccos\left( -\frac{1}{4} \right)$ & 0.9117382909685 \\
$3 \to 2$ & $\frac{1}{2} \arccos\left( -\frac{1}{3} \right)$ & 0.9553166181245\\
$2 \to 1$ & $\pi / 3$ & 1.0471975511966 \\
$1 \to 0$ & $\pi$ & 3.1415926535898
\end{tabular}
\end{center}
\label{table:S3}
\end{table}

In Table~\ref{table:S4}, some of the smaller values of the Coxeter upper bound cannot be computed due to the properties of the Schl\"afli function: it assumes that we have enough spheres to put them in the vertices of a regular $n$--simplex. Thus, for large $r$, we have too few points, and this bound is not applicable. 

\begin{table}[ht]
\caption{Bounds for the kissing number in $\mathbb{S}^4$ }
\begin{center}
\begin{tabular}{ l  c  c  c  c  c}
\hline \vspace*{-0.8em} \\
    & theoretical & lower bound  & SDP  & Levenshtein  & Coxeter\\ 
  r & lower bound & by construction & upper bound &  bound & bound \vspace*{0.1em}\\  \hline \vspace*{-0.8em} \\
$0$ & $ 5.11506 $& $24$~\cite{Musin} & $24.056903$~\cite{Machado}    & $  26 $ & 26.4420\\
$0.064960281031$& $5.0847 $& $ 22$& $ 24.25996 $   & $ 25.8154$ & 26.2614\\
$0.135$& $ 4.98499 $& $ 21$& $ 23.698995 $   & $  25.2181 $ & 25.6681\\
$0.2348312007464$& $4.72978 $& $ 21$& $ 22.343847 $   & $ 23.7439 $ &  24.1511\\
$0.315$ & $ 4.43922 $& $ 20$& $ 20.975086 $   & $  22.1389 $&  22.4263\\
$0.3478604258810$& $4.30116 $& $ 20$& $20.418654  $   & $ 21.3944 $ & 21.6076\\
$0.3743605576995$& $ 4.18278$& $ 18$& $ 20.039183 $   & $ 20.7611 $ & 20.9061\\
$0.393$& $ 4.09608 $& $ 17$& $ 19.493801$   & $ 20.2984$ & 20.3925\\
$0.3966966954949$& $4.07857 $& $ 17$& $ 19.336889 $   & $ 20.2050 $ & 20.2888\\
$0.439$& $ 3.87137 $& $ 16$& $ 17.528082$   & $ 18.7761  $ & 19.0626\\
$0.44269036900123$& $ 3.85274$& $ 16$& $ 17.387671 $   & $ 18.6476 $ & 18.9525\\
$0.49$& $  3.60742 $& $ 15$& $15.92363 $   & $  17.0608$ & 17.5025\\
$0.49969620570817$& $ 3.55583$& $ 15$& $ 15.650850 $   & $ 16.7492  $ & 17.1977\\
  $0.53$& $  3.3923 $& $ 14$& $ 14.877753$   & $  15.8038  $ & 16.2314\\
$0.54100885503509$& $ 3.33217$& $ 14$& $  14.632380$   & $ 15.4706 $ & 15.8766\\
$0.55183$& $ 3.27277 $& $ 13$& $ 14.402314$   & $ 15.1480  $ & 15.5262\\
$0.55558271937072$& $3.2521 $& $ 13$& $ 14.313536 $   & $ 15.0373 $ & 15.4043\\
$0.595$& $ 3.03363  $& $ 12$& $12.970691 $   & $  13.8506$ & 14.1160\\
$0.61547970865277$& $ 2.91955$& $ 12$& $ 12.302214 $   & $ 13 + 7\cdot 10^{-10}  $ & 13.4436\\
$0.6299$& $ 2.8392 $& $ 11$& $ 11.902489$   & $ 12.4471  $ & 12.9700\\
$0.63337378793619$& $ 2.81986$& $ 11$& $ 11.780786 $   & $  12.3190$ & 12.8560\\
$0.653$& $  2.71075  $& $ 10$& $ 10.99092$   & $  11.6302  $ & 12.2128\\
$0.68471920300192$& $ 2.53556$& $ 10$& $ 10.000004 $   & $ 10.6250 $ & 11.17937\\
%$0.6847193$& $2.53556  $& $ 9$& $ 9.99999994 $   & $  10.624998 $ & 11.17936\\
%$0.68811601660265$& $ 2.51692$& $ 9$& $  9.8530813$   & $ 10.5242 $ & 11.0693\\
%$0.71$& $  2.3976 $& $ 8$& $ 8.9684325$   & $9.9017 $ & 10.3642\\
%$0.785398163397449$& $ 2$& $ 8$& $8.0000293  $   & $ 8 - 10^{-14} $ & 8.00008\\
%$0.78539828$& $ 2$& $ 5$& $7.9999982 $   & $ 7.999994 $ & 8.00009\\
%$0.88607712356268$& $  1.52096$& $ 5$& $ 5.008075 $   & $5 + 5\cdot 10^{-9} $& 5.00003\\
%$0.9$& $ 1.46106 $& $ 4$& $ 4.5958861$   & $ 4.4014 $ & 4.5586\\
%$0.91173828638360$& $ 1.4121  $& $ 4$& $ 4.0000002 $   & $ 4 + 2\cdot 10^{-7} $ & 4.0002\\
%$0.9206$& $ 1.37591  $& $ 3$& $3.74363 $   & $ 3.7436$ & $*$\\
%$0.95531661577188$& $  1.2431  $& $ 3$& $ 3+5\cdot 10^{-8} $   & $3 + 4 \cdot 10^{-8}$ & $*$\\
%$\pi/3$& $ 1 $& $ 2$& $ 2+10^{-15}$   & $ 2 $ & $*$
\end{tabular}
\end{center}
\label{table:S4}
\end{table}

\begin{table}[ht]
\caption*{\textsc{Table 4 (continued)}. Bounds for the kissing number in $\mathbb{S}^4$ }
\begin{center}
\begin{tabular}{ l  c  c  c  c  c}
\hline \vspace*{-0.8em} \\
    & theoretical & lower bound  & SDP  & Levenshtein  & Coxeter\\ 
  r & lower bound & by construction & upper bound &  bound & bound \vspace*{0.1em}\\  \hline \vspace*{-0.8em} \\
$0.6847193$& $2.53556  $& $ 9$& $ 9.99999994 $   & $  10.624998 $ & 11.17936\\
$0.68811601660265$& $ 2.51692$& $ 9$& $  9.8530813$   & $ 10.5242 $ & 11.0693\\
$0.71$& $  2.3976 $& $ 8$& $ 8.9684325$   & $9.9017 $ & 10.3642\\
$0.785398163397449$& $ 2$& $ 8$& $8.0000293  $   & $ 8 - 10^{-14} $ & 8.00008\\
$0.78539828$& $ 2$& $ 5$& $7.9999982 $   & $ 7.999994 $ & 8.00009\\
$0.88607712356268$& $  1.52096$& $ 5$& $ 5.008075 $   & $5 + 5\cdot 10^{-9} $& 5.00003\\
$0.9$& $ 1.46106 $& $ 4$& $ 4.5958861$   & $ 4.4014 $ & 4.5586\\
$0.91173828638360$& $ 1.4121  $& $ 4$& $ 4.0000002 $   & $ 4 + 2\cdot 10^{-7} $ & 4.0002\\
$0.9206$& $ 1.37591  $& $ 3$& $3.74363 $   & $ 3.7436$ & $*$\\
$0.95531661577188$& $  1.2431  $& $ 3$& $ 3+5\cdot 10^{-8} $   & $3 + 4 \cdot 10^{-8}$ & $*$\\
$\pi/3$& $ 1 $& $ 2$& $ 2+10^{-15}$   & $ 2 $ & $*$
\end{tabular}
\end{center}
%\label{table:S4}
\end{table}

\section{Average kissing number}\label{section:avg-kissing-number}

Just as in the Euclidean space $\mathbb{R}^n$, one may define the average kissing number of $\mathbb{X}^n = \mathbb{H}^n$ or $\mathbb{S}^n$, for $n\geq 2$. Let $S_1$, $S_2$, $\ldots$, $S_k$ be a collection of spheres in $\mathbb{X}^n$ with disjoint interiors. Let $\mathcal{G} = \mathcal{G}(S_1, \ldots, S_k)$ be its \textit{contact graph}: the vertices of $\mathcal{G}$ are $S_1$, $\ldots$, $S_k$, while there is an edge connecting $S_i$ to $S_j$, $1 \leq i\neq j \leq k$, whenever $S_i$ and $S_j$ come in contact. Let $V = V(\mathcal{G})$ be the vertices of $\mathcal{G}$, and $E = E(\mathcal{G})$ be its edges. Then, the average kissing number is 
\begin{equation*}
\kappa^{avg}_X(n) = \sup_{\mathcal{G}} \,\, \frac{2 |E(\mathcal{G})|}{|V(\mathcal{G})|},
\end{equation*}
where $X = H$ or $S$, depending on whether we consider $\mathbb{H}^n$ or $\mathbb{S}^n$, and the supremum is taken over all possible contact graphs $\mathcal{G}$ for the sphere configurations in $\mathbb{X}^n$. We put no subscript for the Euclidean average kissing number $\kappa^{avg}(n)$ in $\mathbb{R}^n$. 

By using the stereographic projection, it is possible to project a sphere configuration in $\mathbb{S}^n$ to a sphere configuration in $\mathbb{R}^n$, and back. Here, the only condition that needs to be satisfied is that the ``North Pole'' of the projection (the point of $\mathbb{S}^n$ that has no image in $\mathbb{R}^n$) does not belong to any sphere or its interior. This means that the contact graphs in $\mathbb{R}^n$ and in $\mathbb{S}^n$ are in bijective correspondence. 

Analogously, any sphere configuration in $\mathbb{R}^n$ can be placed into the upper half-space model for $\mathbb{H}^n$, and thus it will give a sphere configuration in the latter (since spheres in $\mathbb{H}^n$ can be described by quadratic equations as Euclidean spheres, although their ``hyperbolic'' and ``Euclidean'' centres do not generally coincide). On the other hand, a sphere configuration in the upper half-space model of $\mathbb{H}^n$ can be interpreted as a sphere configuration in a half-space of $\mathbb{R}^n$. Thus, the contact graphs in $\mathbb{R}^n$ and in $\mathbb{H}^n$ are also in bijective correspondence. 

From the above noted correspondences, it follows that 
\begin{equation*}
\kappa^{avg}_H(n) = \kappa^{avg}_S(n) = \kappa^{avg}(n),
\end{equation*}
for all $n\geq 2$. For the lower and upper bounds on $\kappa^{avg}(n)$, we refer the reader to \cite{DKO} and the references therein. 

\FloatBarrier


\begin{thebibliography}{100}
\normalsize

\bibitem{BachocVallentin}\textsc{C. Bachoc and F. Vallentin}, \textit{New upper bounds for kissing numbers from semidefinite programming},  Journal Amer. Math. Soc. \textbf{21} (2008), 909--924.

\bibitem{Bezdek1}\textsc{K. Bezdek}, \textit{Contact numbers for congruent sphere packings in Euclidean $3$--space}, Discrete Comput. Geom. \textbf{48} (2012), 298--309.

\bibitem{Bezdek-Reid}\textsc{K. Bezdek and S. Reid}, \textit{Contact graphs of unit sphere packings revisited}, J. Geometry \textbf{104} (2013), 57--83.

\bibitem{Bor}\textsc{K. B\"or\"oczky}, \textit{Packing of spheres in spaces of constant curvature}, Acta Math. Acad. Sci. Hungar. \textbf{32} (1978), 243--261.

\bibitem{B}\textsc{L. Bowen}, \textit{Circle packing in the hyperbolic plane}, Math. Physics Electronic J. \textbf{6} (2000), 1--10.

\bibitem{C}\textsc{C. Chabauty}, \textit{R\'esultats sur l'empilement de calottes  \'egales sur une p\'erisph\`ere de $\mathbb{R}^n$ et correction\`a un travail ant \'erieur}, C. R. Acad. Sci. Paris \textbf{236} (1953), 1462--1464.

\bibitem{Casselman}\textsc{B. Casselman}, \textit{The difficulties of kissing in three dimensions},  Notices Amer. Math. Soc. \textbf{51} (2004), 884--885.

\bibitem{Coxeter}\textsc{H.S.M. Coxeter}, \textit{An upper bound for the number of equal non-overlapping spheres that can touch another of the same size}, Proc. of Symp. in Pure Math. AMS \textbf{7} (1963), 53--71.

\bibitem{Danzer}\textsc{L. Danzer}, \textit{Finite point-sets on 
$\mathbb{S}^2$ with minimum distance as large as possible}, Discr. Math. \textbf{60} (1986), 3--66.

\bibitem{DelsarteGoethalsSeidel}\textsc{P. Delsarte, J.M. Goethals, and J.J. Seidel}, \textit{Spherical codes and designs},  Geom. Dedicata \textbf{6} (1977), 363--388.

\bibitem{DKO}\textsc{M. Dostert, A. Kolpakov, and F.M. de Oliveira Filho}, \textit{Semidefinite programming bounds for the average kissing number}, Israel J. Math. (to appear); \texttt{arXiv:2003.11832}

\bibitem{Dostert}\textsc{M. Dostert, D. de Laat, and P. Moustrou}, \textit{Exact semidefinite programming bounds for packing problems}, SIAM J. Optim. (to appear); \texttt{arxiv.org/abs/2001.00256}

\bibitem{Toth-RegularFigures}\textsc{L. Fejes T\'oth}, \textit{Regular figures}, International Series of Monographs on Pure and Applied Mathematics \textbf{48}. Oxford: Pergamon Press, 1964.

\bibitem{FejesToth}\textsc{L. Fejes T\'oth}, \textit{\"Uber die Absch\"atzung des k\"urzesten Abstandes zweier Punkte eines auf einer Kugelfl\"ache liegenden Punktsystems}, Jber. Deutch. Math. Verein \textbf{53} (1943), 66--68.

\bibitem{FT}\textsc{L. Fejes T\'oth}, \textit{\"Uber die dichteste Horozyklenlagerung}, Acta Math. Acad. Sci. Hungar. \textbf{5} (1954), 41--44.

\bibitem{Yaglom}\textsc{L. Fejes T\'oth}, \textit{Lagerungen in der Ebene, auf der Kugel und im Raum}, Springer-Verlag, 1953. Russian translation, Moscow, 1958, with commentaries and appendix by I.~M.~Yaglom.
 
\bibitem{Glazyrin}\textsc{A.  Glazyrin}, \textit{Contact graphs of ball packings}, J. Combin. Theory Ser. B  \textbf{145} (2020), 323--340.

\bibitem{Sloane-et-al}\textsc{R.H. Hardin, W.D. Smith, N.J.A. Sloane, et al.} \textit{Spherical Codes}, \url{http://neilsloane.com/packings/}

\bibitem{JJP}\textsc{M. Jenssen, F. Joos, and W. Perkins}, \textit{On kissing numbers and spherical codes in high dimensions}, Advances Math. \textbf{335} (2018), 307--321. 

\bibitem{KL}\textsc{G.A. Kabatiansky and V.I. Levenshtein}, \textit{Bounds for packings on a sphere and in space},  Problems Inform. Transmission \textbf{14} (1978), 1--17.

\bibitem{Levenshtein}\textsc{V.I.  Levenshtein}, \textit{On bounds  for packing  in $n$-dimensional  Euclidean  space},  Soviet Math. Dokl. \textbf{20} (1979), 417--421. 

\bibitem{Levenshtein-2}\textsc{V.I.  Levenshtein}, \textit{Bounds for packings in metric spaces and some applications},  Problemy Kibernetiki \textbf{40} (1983), 43--110 (in Russian). 
 
\bibitem{Machado}\textsc{F.C. Machado and F.M. de Oliveira Filho}, \textit{Improving the semidefinite programming bound for the kissing number by exploiting polynomial symmetry}, Experimental Math. \textbf{27} (2018), 362--369.

\bibitem{MittelmannVallentin}\textsc{H.D. Mittelmann and F. Vallentin}, \textit{High-accuracy semidefinite programming bounds for kissing numbers},  Experimental Math. \textbf{19} (2010), 175--179.

\bibitem{Musin}\textsc{O.R. Musin}, \textit{The kissing number in four dimensions},  Ann. Math. \textbf{168} (2008), 1--32.

\bibitem{Musin2}\textsc{O.R. Musin}, \textit{Towards a proof of the $24$-cell conjecture}, Acta Math. Hungar. \textbf{155} (2018), 184--199.

\bibitem{Odlyzko79newbounds}\textsc{A.M. Odlyzko and N.J.A. Sloane}, \textit{New bounds on the number of unit spheres that can touch a unit sphere in $n$ dimensions}, J. Combin. Theory Ser. A \textbf{26} (1979), 210--214.

\bibitem{Ratcliffe}\textsc{J.G. Ratcliffe},  \textit{Foundations of hyperbolic manifolds}, Graduate Texts in Mathematics \textbf{149}. New--York: Springer, 2013.

\bibitem{Shannon}\textsc{C.E. Shannon}, \textit{Probability of error for optimal codes in a Gaussian channel}, Bell System Tech. J. \textbf{38} (1959), 611--656.

\bibitem{SchuetteWaerden}\textsc{K. Sch\"utte and B.L. van der Waerden}, \textit{Das Problem der dreizehn Kugeln},  Math. Ann. \textbf{125} (1953), 325--334.

\bibitem{sage}\textsc{W.A. Stein, et al.}, \textit{Sage Mathematics Software (Version 6.6)}, The Sage Development Team, 2012, \url{http://www.sagemath.org} 

\bibitem{Szirmai1}\textsc{J. Szirmai}, \textit{The densest geodesic ball packing by a type of $Nil$ lattices}, Beitr\"age zur Algebra und Geometrie \textbf{48} (2007), 383--397.

\bibitem{Szirmai2}\textsc{J. Szirmai}, \textit{Geodesic ball packings in $\mathbb{S}^2 \times \mathbb{R}$ space for generalized Coxeter space groups}, Beitr\"age zur Algebra und Geometrie \textbf{52} (2011), 413--430.

\bibitem{Szirmai3}\textsc{J. Szirmai}, \textit{Horoball packings to the totally asymptotic regular simplex in the hyperbolic $n$--space}, Aequationes Math. \textbf{85} (2013), 471--482.

\bibitem{W}\textsc{A.D. Wyner},  \textit{Capabilities of bounded discrepancy decoding}, Bell Syst. Tech. J.  \textbf{44} (1965), 1061--1122.

\end{thebibliography}
\end{document}